\documentclass{amsart}
\usepackage{amsmath,amsfonts, amsthm,amssymb,enumerate,comment}

\newtheorem{theorem}{Theorem}
\newtheorem{proposition}{Proposition}
\newtheorem{example}{Example}
\newtheorem{lemma}{Lemma}
\newtheorem{corollary}{Corollary}
\newtheorem{remark}{Remark}
\newtheorem{definition}{Definition}

\newcommand{\R}{\mathbb{R}}

\newcommand{\sign}{\mathrm{sign}}

\newcommand{\E}{\mathbb{E}}
\newcommand{\Prob}{\mathbb{P}}

\newcommand{\sL}{\mathcal{L}}
\newcommand{\sB}{\mathcal{B}}
\newcommand{\sF}{\mathcal F}
\newcommand{\sI}{\mathcal I}
\newcommand{\sC}{\mathcal C}
\newcommand{\sG}{\mathcal{G}}
\newcommand{\sV}{\mathcal{V}}

\begin{document}

\title{Integrability of solutions of the Skorokhod Embedding Problem}

\date{\today}
\author{David Hobson}

\begin{abstract}

Suppose $X$ is a time-homogeneous diffusion on an interval $I^X \subseteq \R$ and 
let $\mu$ be a probability measure on $I^X$. Then $\tau$ is a solution of the 
Skorokhod embedding problem (SEP) for $\mu$ in $X$ if $\tau$ is a stopping time and 
$X_\tau \sim \mu$.

There are well-known conditions which determine whether there exists a solution 
of the SEP for $\mu$ in $X$. We give necessary and sufficient conditions for 
there to exist an integrable solution. Further, if there exists a solution of 
the SEP then there exists a minimal solution. We show that every minimal 
solution of the SEP has the same first moment.

When $X$ is Brownian motion, every integrable embedding of $\mu$ is minimal.  
However, for a general diffusion there may be integrable embeddings which are 
not minimal.

\end{abstract}

\maketitle


\section{Introduction}
\label{sec:intro}

Let $X$ be a regular, time-homogeneous diffusion on an interval $I^X\subseteq \R$, with $X_0 = x \in int(I^X)$, and let $\mu$ be a probability measure on $\overline{I^X}$. Then $\tau$ is a solution of the Skorokhod embedding problem (Skorokhod~\cite{Skorokhod:65}) for $\mu$ in $X$ if $\tau$ is a stopping time and $X_\tau \sim \mu$.
We call such a stopping time an embedding (of $\mu$ in $X$).

For a general Markov process 
Rost~\cite{Rost:71} gives necessary and sufficient conditions which determine whether a solution to the Skorokhod embedding problem (SEP) exists for a given target law.
The conditions are expressed in terms of the potential. When applied to Brownian motion (where we include the case of Brownian motion on an interval subset of $\R$, provided the process is absorbed at finite endpoints)
these conditions lead to a characterisation
of the set of measures which can be embedded in Brownian motion.
Then, in the case of a regular, one-dimensional, time-homogeneous diffusion with absorbing endpoints, necessary and sufficient conditions for the existence of a solution to the SEP can be derived via a change of scale. Let $s$ be the scale function of $X$; then $Y=s(X)$ is a local martingale, and in particular a time-change of Brownian motion. Further, let $I=s(I^X)$ be the state space of $Y$. Then the set of measures for which a solution of the SEP exists depends on both $I$ and the relationship between the starting value of $Y$ and the mean of the image under $s$ of the target law, see Theorem~\ref{thm:existence} below.

Apart from the existence result above, most of the literature on the SEP has concentrated on the case where $X$ is Brownian motion in one dimension. Exceptions include Rost~\cite{Rost:71} as mentioned above, Bertoin and LeJan~\cite{BertoinLeJan:92} who consider embeddings in any time-homogeneous process with a well-defined local time, Grandits and Falkner~\cite{GranditsFalkner:00} (drifting Brownian motion), Hambly {\em et al}~\cite{HamblyKerstingKyprianou:02} (Bessel process of dimension 3) and Pedersen and Peskir~\cite{PedersenPeskir:01} and Cox and Hobson~\cite{CoxHobson:04} (these last two consider embeddings in a general time-homogenenous diffusions).

In the Brownian setting many solutions of the SEP have been described; see Obloj~\cite{Obloj:04} or Hobson~\cite{Hobson:11} for a survey. Given there are many solutions, it is possible to look for criteria which characterise `small' or `good' solutions. In both the Brownian case and more generally, there is a natural class of good solutions of the SEP, namely the minimal embeddings (Monroe~\cite{Monroe:72}). An embedding $\tau$ is minimal if whenever $\sigma \leq \tau$ is another embedding (of $\mu$ in $X$) then $\sigma = \tau$ almost surely.

Another criteria for a good solution might be that it is integrable, or as small as possible in the sense of expectation. In this article we are interested in the integrability or otherwise of solutions of the SEP, and the relationship between integrability and minimality in the case where $X$ is a time-homogeneous diffusion in one dimension.

Consider the case where $X$ is Brownian motion null at zero and write $W$ for $X$. By the results of
Rost~\cite{Rost:71}
there exists a solution of the SEP for $\mu$ in $W$ on $\R$ for {\em any} measure $\mu$ on $\R$. If we require integrability of the embedding then the story is also well-known:

\begin{theorem}
[Monroe~\cite{Monroe:72}]
\label{thm:brownian}
There exists an integrable solution of the SEP for $\mu$ in $W$ if and only if $\mu$ is centred and in $L^2$.
 Further, in the case of centred square-integrable target measures, $\tau$ is minimal for $\mu$ if and only if $\tau$ is an embedding of $\mu$ and $\E[\tau] < \infty$.
\end{theorem}

Our goal in this paper is to consider the case where $X$ is a regular
time-homogeneous diffusion on an interval $I^X$
with absorbing endpoints. Let $x \in int(I^X)$ denote the initial value of $X$, let $m^X$ denote the speed measure, and $s^X$ the scale function. Let $\mu$ be a probability measure on $\overline{I^X}$.

Our main result is as follows:
\begin{theorem}
\label{thm:general}
There exists an integrable solution of the SEP for $\mu$ in $X$ if and only if $E_X(x;\mu)<\infty$ where $E_X(x;\mu)$ is defined in (\ref{eqn:EXdef}) below.
Further, in the case where $E_X(x;\mu)<\infty$ then $\tau$ is minimal for $\mu$ if and only if $\tau$ is an embedding and $\E[\tau] = E_X(x;\mu)$.
\end{theorem}

In the Brownian case there is a dichotomy, and for any embedding either $\E[\tau] = \int x^2 \mu(dx)$ or $\E[\tau] = \infty$, and so if the target law is square integrable then minimality of an embedding is equivalent to integrability. This is not true in general for diffusions: we can have integrable embeddings which are not minimal. The converse is also true: both in the Brownian case and more generally we can have minimal embeddings which are not integrable. This will be the case if $E_X(x;\mu)=\infty$.

We close the introduction by considering a quartet of illuminating and 
motivating examples.

\begin{example}
\label{ex:ABM}
Let $Z = (Z_t)_{t \geq 0}$ be Brownian motion on $\R_+$ absorbed at zero, and with $Z_0 = z>0$. Then there exists an embedding of $\mu$ if and only if $\int x \mu(dx) \leq z$. Moreover, there exists an integrable embedding of $\mu$ in $Z$ if and only if $\int x \mu(dx) = z$ and $\int x^2 \mu(dx) < \infty$ and then an embedding $\tau$ is minimal if and only if $\E[\tau] < \infty$ if and only if $\E[\tau] = \int (x-z)^2 \mu(dz)$. Note that $Z$ is a supermartingale so the necessity of $\int x \mu(dx) \leq z$ is clear.
\end{example}

\begin{example}
\label{ex:DBM}
Let $V = (V_t)_{t \geq 0}$ be upward drifting Brownian motion with $V_0 = v$. 
In particular, suppose $V$ solves $V_t = v + a W_t + b t$ with $b>0$ and 
$W_0=0$, and set 
$\beta = 2b/a^2$. Then there exists an embedding of $\mu$ if and only if $\int e^{- \beta(u-v)}\mu(du) \leq 1$. (Upward drifting Brownian motion is transient to $+\infty$ and so there will be an embedding of $\mu$ provided $\mu$ does not place too much mass at values far below $v$.) Moreover, there exists an integrable embedding of $\mu$ if and only if $\int e^{- \beta(u-v)} \mu(du) \leq 1$ and $\int u^+ \mu(du) < \infty$. If there exists an integrable embedding then an embedding $\tau$ is minimal if and only if $\E[\tau] = E(v;\mu)$ where
\[ E(v;\mu) = \frac{1}{b} \left(\int u \mu(du) - v \right) < \infty. \]
\end{example}

\begin{example}
\label{ex:Bes3}
Let $P= (P_t)_{t \geq 0}$ be a Bessel process of dimension 3 started at $P_0=p>0$.
Then there exists an embedding of $\mu$ if and only if $\int x^{-1}
\mu(dx) \leq p^{-1}$. Moreover, there exists an integrable embedding of $\mu$ if and only if $\int x^{-1}
\mu(dx) \leq p^{-1}$ and $\int x^2 \mu(dx) < \infty$ and then an embedding $\tau$ is minimal for $\mu$ if and only if $\tau$ is an embedding and $\E[\tau] = E(p;\mu)$ where
\begin{equation}
E(p;\mu)= \frac{1}{3} \int x^2 \mu(dx) - \frac{p^2}{3}
\end{equation}
Note that a Bessel process is transient to infinity, and so for there to exist an embedding of $\mu$, $\mu$ cannot place too much mass near zero. For an integrable embedding then in addition we cannot have too much mass far from zero as the process takes a long time to get there. Note also that $Y=P^{-1}$ is a diffusion in natural scale and that $Y$ is the classical Johnson-Helms example of a local martingale which is not a martingale.

The results extend to the case $p=0$. Then any $\mu$ on $\R^+$ can be embedded in $P$. There exists an integrable embedding if and only if $\mu$ is square integrable.
\end{example}

\begin{example}
\label{ex:intnonmin}
Let $Q= (Q_t)_{t \geq 0}$ solve $dQ_t = (1 + Q_t^2) dW_t$ subject to $Q_0=0$. 
Let $\mu = \frac{1}{2} \delta_1 + \frac{1}{2} \delta_{-1}$. Let $\tau = \max 
[ \inf \{ u : Q_u = - 1 \} , \inf \{ u : Q_u = 1 \} ]$. Then $\tau$ is an 
embedding of $\mu$ and $\tau$ is integrable, but $\tau$ is not minimal.
\end{example} 

\section{Preliminaries, notation and the switch to natural scale}
Let $X$ be a time-homogeneous diffusion with state space $I^X$, started at $x 
\in int(I^X)$, and suppose that if $X$ can reach an endpoint of $I^X$, then 
such an endpoint is absorbing.
Suppose that $X$ is regular, ie for all $x' \in int(I^X)$ and $x'' \in I^X$, $\Prob^{x'}( H_{x''} < \infty ) > 0$. Then, see Rogers and Williams~\cite{RogersWilliams:00b} or Borodin and Salminen~\cite{BorodinSalminen:02},  $X$ has a scale function $s$ and $Y = s(X)$ is a diffusion in natural scale on the interval $I=
s(I^X)$. Denote the endpoints of $I$ by $\{\ell,r\}$ and suppose $y = s(x)$ lies in $(\ell,r)$. Then we have
$-\infty  \leq \ell < y < r \leq \infty$.

For a diffusion process $Z$ let $H^Z_z = \inf \{ s \geq 0 : Z_s = z \}$, and $H^Z_{a,b} = H^Z_a \wedge H^Z_b$. Where the process $Z$ involved is clear, the superscript may be dropped.

We have that
$(Y_{t \wedge H^Y_{\ell,r}})_{t \geq 0}$ is a continuous local martingale. In particular, we can write $Y_t = W_{\Gamma_t}$ for some Brownian motion $W$ started at $y$ and a strictly increasing time-change $\Gamma$. We have already seen from Example~\ref{ex:Bes3} that $Y$ may easily be a strict local martingale.

Let $\mu$ be a law on $\overline{I^X}$ and define $\nu = \mu \circ s^{-1}$ so that for a Borel subset of $\overline{I}$, $\nu(A) = \mu ( s^{-1}(A))$. Then $\tau$ is an embedding of $\mu$ in $X$ if and only if $\tau$ is an embedding of $\nu$ in $Y$. Moreover, the integrability of $\tau$ is also unaffected by a change of scale, and thus we lose no generality in assuming that our diffusion is in natural scale. Minimality is another property which is preserved under a change of scale.

Henceforth, therefore, we assume we are given a local martingale diffusion $Y$ 
on $I$ with $Y_0=y \in int(I)$ and target measure $\nu$ on $\overline{I}$. 
Provided $\nu \in L^1$, write $\overline{\nu}$ for the mean of $\nu$, with a 
similar convention for other measures. It follows from our assumption on $X$ 
that if $Y$ 
can reach an 
endpoint $\ell$ or $r$ of $I$ in finite time then that endpoint is absorbing. 
The diffusion $Y$ in natural scale is characterised by its speed measure which 
we denote by $m$. Recall that if $Y$ solves the SDE $dY_t = \eta(Y_t) dB_t$ for 
a continuous diffusion coefficient $\eta$ then $m(dy) = dy/\eta(y)^{2}$.

\begin{theorem}[Pedersen and Peskir~\cite{PedersenPeskir:01}, Cox and Hobson~\cite{CoxHobson:04}]
\label{thm:existence}
\begin{enumerate}
\item[{\rm (i)}]
Suppose $I$ is a finite interval. Then $\nu$ can be embedded in $Y$ if and only if $y = \int x \nu(dx)$.
\item[{\rm (ii)}]
Suppose $I=(\ell, \infty)$ or $[\ell,\infty)$ for $\ell > - \infty$. Then $\nu$ can be embedded in $Y$ if and only if $y \geq
\int x \nu(dx)$.
\item[{\rm (iii)}]
Suppose $I=(-\infty,r)$ or $(-\infty,r]$ for $r<\infty$.
Then $\nu$ can be embedded in $Y$ if and only if $y \leq
\int x \nu(dx)$.
\item[{\rm (iv)}]
Suppose $I=\R$.
Then $\nu$ can be embedded in $Y$ if and only if $\nu$ is a measure on $\R$.
\end{enumerate}
\end{theorem}

The idea behind the proof is to write $Y$ as a time-change of Brownian motion, $Y_t = W_{\Gamma_t}$. Then, since $Y$ is absorbed at the endpoints we must have that $\Gamma_t \leq H^W_{\ell,r}$ for each $t$.

In the first case of the theorem $Y$ is a bounded martingale and $\E[Y_\tau] = y$ for any $\tau$. In the second case $Y$ is a local martingale bounded below and hence a supermartingale for which $\E[Y_\tau] \leq y$. In the third case $Y$ is a submartingale.

\begin{proposition}
\label{prop:minimal}
Suppose that at most one endpoint of $I$ is infinite. Then any embedding of $\nu$ on $int(I)$ is minimal.
\end{proposition}

\begin{proof}
We prove the result in the case $I=(\ell,\infty)$ or $[\ell,\infty)$ with $\ell > -\infty$. The other cases are similar.

Since $I$ has a finite endpoint, $Y$ is transient. Further, $Y$ is a supermartingale.

Let $\tau$ be an embedding of $\nu$ where $\overline{\nu} \leq y$. Let $\sigma \leq \tau$ be another embedding. Then, from the supermartingale property,
$\E[Y_\tau ; Y_\sigma \leq x] \leq \E[Y_\sigma; Y_\sigma \leq x]$ and since $Y_\sigma$ and $Y_\tau$ are equal in law,
\[ \E[x -Y_\tau ; Y_\sigma \leq x] \geq \E[x -Y_\sigma; Y_\sigma \leq x] = \E[x - Y_\tau ; Y_\tau \leq x] = \sup_{A} \E[x - Y_\tau;A] \]
Then, modulo null sets $(Y_\tau \leq x) = (Y_\sigma \leq x)$ and hence $Y_\sigma = Y_\tau$ almost surely.

Suppose $\sigma \leq \eta \leq \tau$. Then
\[ Y_\eta \geq \E[Y_\tau | \sF_\eta] = \E[Y_\sigma | \sF_\eta] = Y_\sigma,  \]
almost surely. But also $\E[Y_\eta - Y_\sigma] \leq 0$ since $Y$ is a 
supermartingale, and hence $Y_\eta = Y_\sigma$ almost surely. It follows that 
$Y$ is almost surely constant over the interval $[\sigma,\tau]$. But $Y$ is a 
time change of Brownian motion $Y_t = W_{\Gamma_t}$ for some strictly increasing 
time-change $\Gamma$. Brownian motion has no intervals of constancy, and hence 
nor does $Y$. It follows that $\sigma = \tau$ almost surely and hence $\tau$ is 
minimal.

\end{proof}

We close this section with a discussion of the Brownian case, including a partial proof of Theorem~\ref{thm:brownian}, followed by a discussion of the local martingale diffusion case.

For $W$ a Brownian motion null at 0, $W^2_{t \wedge \tau} - (t \wedge \tau)$ is a martingale and
\begin{equation}
\E[\tau] = \liminf \E[t \wedge \tau] =  \liminf \E[W^2_{t \wedge \tau}] \geq \E[\liminf W^2_{t \wedge \tau}] = \E[W_\tau^2].
\end{equation}
Moreover, from
Doob's $L^2$ submartingale inequality
we know that
$\E[\tau]<\infty$ if and only if $\E[(W_\tau^*)^2]<\infty$, and then $(W_{t \wedge \tau})_{t \geq 0}$ and $(W_{t \wedge \tau}^2)_{t \geq 0}$ are uniformly integrable. It follows that if $\E[\tau]<\infty$ then
\[ 0 = \lim \E[ W_{t \wedge \tau}] = \E[W_\tau] = \int x \mu(dx) \]
and
\[ \E[\tau] = \lim \E[t \wedge \tau] = \lim \E[ W_{t \wedge \tau}^2 ] = \E[W_\tau^2] = \int x^2 \mu(dx),\]
so that $\mu$ is centred and in $L^2$.

Conversely, if $\mu$ is centred and in $L^2$ then there are several classical constructions which realise an integrable embedding, including those of Skorokhod~\cite{Skorokhod:65} and Root~\cite{Root:69}.
See Obloj~\cite{Obloj:04} or Hobson~\cite{Hobson:11} for a discussion.

The final statement of Theorem~\ref{thm:brownian} is deeper, and follows from Theorem 5 of Monroe~\cite{Monroe:72}.
One of the main goals of this work is to extend the work of Monroe to general diffusions. Note that the arguments above yield that in the Brownian case if $\tau$ is an embedding of $\mu$ and $\E[\tau]<\infty$ then $\E[\tau] = \int x^2 \mu(dx)$, so that if $\mu$ is centred and in $L^2$ then every integrable embedding is minimal.

Consider now the case of a general diffusion $Y$
in natural scale. Suppose $Y_0=y=0$ and that $\nu$ is centred. Then to determine whether there might exist a integrable embedding we might expect to replace the condition $\int x^2 \mu(dx)<\infty$ of the Brownian case with some other integral test depending on the speed measure $m$ of $Y$ and the target measure $\nu$. Indeed we find this is the case with $x^2$ replaced by a convex function $q$ defined in (\ref{eqn:qdef}) in the next section.

But what if $\nu$ is not
centred? In the Brownian case there is no hope that the target law can be embedded in integrable time, not least because $\E[H^W_x] = \infty$ for each non-zero $x$, but what if $Y$ is some other diffusion?

Suppose the state space $I$ of $Y$ is unbounded above. Suppose $Y_0=y$ and $\nu \in L^1$ with $\overline{\nu} = \int x \nu(dx) < y$. (In this discussion we exclude the degenerate case where $Y$ is a point mass at $\ell$.) One candidate way to embed $\nu$ is to first wait until $H^Y_{\overline{\nu}} = \inf \{ t : Y_t = \overline{\nu} \}$ and then to embed $\nu$ in $Y$ started at $\overline{\nu}$, ie to set
\begin{equation}\label{eqn:taudef}
  \tau = H^Y_{\overline{\nu}} + \tau^{\overline{\nu}, \nu} \circ \Theta_{H^Y_{\overline{\nu}}}
\end{equation}
where $\Theta$ is the shift operator $\Theta_t ( \omega( \cdot)) = \omega(t+ \cdot)$ and $\tau^{\overline{\nu},\nu}$ is some embedding of $\nu$ in $Y$ started at $\overline{\nu}$. Note that since $I$ is unbounded above and $Y$ is a time-change of Brownian motion, it follows that $ H^Y_{\overline{\nu}}$ is finite almost surely. The embedding in (\ref{eqn:taudef}) will be integrable if {\em both} $H^Y_{\overline{\nu}}$ {\em and} $\tau^{\overline{\nu}, \nu}$ are integrable, and we can decide if it is possible to choose $\tau^{\overline{\nu}, \nu}$ integrable using the integral test of the centred case. Our results show that although embeddings of $\nu$ need not be of the form given in (\ref{eqn:taudef}), nonetheless there exist integrable embeddings if and only if both $\E[H^Y_{\overline{\nu}}]<\infty$
and there is an integrable embedding $\tau^{\overline{\nu},\nu}$ of $\nu$ in $Y$ started at $\overline{\nu}$.
In that case every minimal embedding has the same first moment.

\section{Every minimal embedding has the same first moment}
Let $Y$ be a regular diffusion in natural scale on $I \subseteq \R$. Suppose $Y_0 = y$. Let $m$ denote the speed measure of $Y$, define $q_u$ via
\begin{equation}\label{eqn:qdef}
  q_u(w) = 2 \int_u^w dv \int_u^v m(dz) = 2 \int_u^w m((u,v))dv
\end{equation}
and let $q = q_y$.
Then $q(Y_t) - t$ is a local martingale, null at zero. 

\begin{definition}
\label{def:E}
If $\nu \notin L^1$ set $E_Y(y;\nu)=\infty$. For $\nu \in L^1$ define
\begin{equation}
\label{eqn:Edefu}
E_Y(y;\nu) = \int q_y(z) \nu(dz) + |y - \overline{\nu}| \lim_{n \to \infty} \frac{ q_y( y + n \sign(y - \overline{\nu})) }{n}
\end{equation}
with the convention that $\sign(0)=0$. 
\end{definition}

In the case of a diffusion in natural scale, the main result of this paper is the following:
\begin{theorem}
\label{thm:martingale}
There exists an integrable solution of the SEP for $\nu$ in $Y$ if and only if $E_Y(y;\nu)<\infty$.
Further, in the case where $E_Y(y;\nu)<\infty$ we have that $\tau$ is minimal for $\nu$ if and only if $\tau$ is an embedding and $\E[\tau] = E_Y(y;\nu)$.
\end{theorem}

Our goal is to prove Theorem~\ref{thm:martingale}. In this section we suppose that $\nu \in L^1$ and $-\infty \leq \ell < y < r \leq \infty$.

\subsection{The centred case with support in a sub-interval}
Suppose $\nu$ is a measure with mean $y$ and support in a subset $[L,R] \subset (\ell,r)$ of $I$ where $L < y < R$.

\begin{lemma}
\label{lem:minimalLR}
Suppose $\tau \leq H_{L,R}$. Then $\tau$ is minimal for $\sL(Y_\tau)$ in $Y$ and $\E[\tau] = \E[q(Y_\tau)]$.
\end{lemma}

\begin{proof}
We have $Y_{t \wedge \tau}$ is bounded and $\E[Y_\tau] = y$. Also $q$ is
bounded on $[L,R]$. Hence
\[ \E[q(Y_\tau)] = \lim_t \E[q(Y_{t \wedge \tau})] = \lim_{t} \E[t \wedge
\tau] = \E[\tau]. \]
In general, from Fatou's Lemma we know that for any embedding $\chi$ of $\nu$,
\[ \E[\chi] = \lim_t \E[ \chi \wedge t]
\geq \lim_t \E[q(Y_{t \wedge \chi})] \geq \E[q(Y_\chi)] =
\int q(x) \nu(dx). \]
Then if $\chi \leq \tau$ and both $\chi$ and $\tau$ are embeddings of $\nu$, we must have $\chi = \tau$ almost surely. Hence $\tau$ is minimal. See also Proposition 4 in \cite{AnkirchnerHobsonStrack:13}.

\end{proof}

Suppose that $\sigma$ is an embedding of $\nu$.
Our goal is to show that there exists an embedding $\tilde{\sigma}$ of $\nu$ such that $\tilde{\sigma} \leq \sigma \wedge H_{L,R}$. Then $\tilde{\sigma}$ is minimal and $\E[\tilde{\sigma}] = \int q(x) \nu(dx)$. It follows that if $\sigma$ is minimal, then $\sigma = \tilde{\sigma}$ and $\E[\sigma]= \int q(x) \nu(dx)$.

 Following a definition of Root~\cite{Root:69}, we define a barrier to be a closed subset $B$ of $G=[0,\infty] \times [-\infty,\infty]$ such that $(\infty, x) \in B$ for all $x \in [-\infty,\infty]$, $(t, -\infty) \cup (t,\infty) \in B$ for all $t \in [0,\infty]$, if $(0,x) \in B$ for $x>y$ then $(0,x') \in B$ for $x'>x$, similarly if $(0,x) \in B$ for $x<y$ then $(0,x') \in B$ for $x'<x$ and finally if $(t,x) \in B$ then $(s,x) \in B$ for all $s>t$. 
Let $\sB$ be the space of all barriers and given $L,R$ with $\ell \leq L < y < R \leq r$ let $\sB_{L,R}$ be the set of all barriers $B$ with $(0,L)$ and $(0,R)$ in $B$, and then $(t,x) \in B$ for $(t \geq 0, x \leq L)$ and $(t \geq 0, x \geq R)$.

Let $\rho$ be the standard Euclidean metric on $\R^2$. 
We map $G$ into a bounded rectangle $F = [0,1] \times [-1,1]$ by $(t,x) \mapsto (t/(1+t),x/(1+|x|))$ and let $r$ be the induced metric on $G$ given by
\[ r((t,x),(s,y)) = \rho \left( \left( \frac{t}{1+t},\frac{x}{1+|x|} \right),\left( \frac{s}{1+s},\frac{y}{1+|y|} \right) \right) . \]
Now define the metric $r_\sG$ on the set $\sG$ of closed subsets of $G$ by 
\[ r_{\sG}(C,D) = \max \left\{ \sup_{(t,x) \in C} r((t,x),D)), \sup_{(s,y) \in D} r((s,y),C)) \right\}; \]
then $\sG$ is a separable compact space
and the spaces $\sB$ and $\sB_{[L,R]}$ are compact. For $B \in \sB$ define
\begin{equation*}
  \tau_{B} = \inf \{ t : (t, Y(t)) \in B \}.
\end{equation*}

\begin{lemma}
\label{lem:m4} Suppose $\nu$ has mean $y$ and support in $[L,R]$. Suppose that $\sigma$ is an embedding of $\nu$.
Then there is a barrier $B \in \sB_{L,R}$ such that $\sigma \wedge \tau_B \leq H_{L,R}$ is a minimal embedding of $\nu$ and $\E[\sigma \wedge \tau_B] = \int q(x) \nu(dx)$.
\end{lemma}

\begin{proof}

First suppose $\nu$ puts mass on a finite subset of points in
$[L,R]$. In this case it is easy to prove the result by adapting the proof in Monroe~\cite{Monroe:72} which is based on topological arguments. We choose instead to give a more probabilistic proof.

Let $\nu$ be a measure on $n+2$ points.
Label the points $y_0 < y_1 < \cdots < y_n < y_{n+1}$.
Let $\sC = \{ b = (b_0, b_1, \ldots, b_{n+1}) \in \R^{n+2}_+ ; b_0 = 0 = b_{n+1} \}$.
Given $b \in \sC$ let $\eta_b$ be the law of $Y_{\tau(b)}$ where
\[ \tau(b) = \inf \{ u > 0 :  Y_u = y_k, u \geq b_k, \mbox{some $k \in \{ 0, 1, \ldots n, n+1 \}$} \} \]
and note that $\eta_b$ is a probability measure on the same points as $\nu$ with mean $y$. Let
\[ \sC_{\leq,\nu} = \{ b \in \sC : \eta_b(\{y_k\}) \leq \nu(\{y_k\}), 1 \leq k \leq n \}. \]

Suppose that $\gamma=(0,\gamma_1, \ldots, \gamma_{n},0)$ and $\lambda=(0,\lambda_1, \ldots \lambda_{n},0)$ are elements of $\sC_{\leq, \nu}$,
and consider $\gamma \wedge \lambda = (0,\gamma_1 \wedge \lambda_1, \ldots \gamma_{n} \wedge \lambda_{n},0)$.
Set $A = \{ k : \gamma_k <\lambda_k \}$. Then for $k \in A$, $\eta_{\lambda \wedge \gamma}(\{y_k\}) \leq \eta_\gamma(\{y_k\}) \leq \nu( \{y_k\})$ and for $k \in \{1, \ldots n \} \setminus A$, $\eta_{\lambda \wedge \gamma}(\{y_k\}) \leq \eta_\lambda(\{y_k\}) \leq \nu( \{y_k\})$. Hence $\gamma \wedge \lambda \in \sC_{\leq,\nu}$.

It follows that $\sC_{\leq, \nu}$ has a minimal element, $\underline{b}$ say and that $\eta_{\underline{b}}(\{y_k\}) \leq \nu( \{y_k \})$ for all $1 \leq k \leq n$. If $\eta_{\underline{b}}(\{y_j\}) < \nu( \{y_j\})$ for some $j$ then
by making the element of $\underline{b}$ with label $j$ smaller we can increase the mass embedded at $j$, without violating the constraint $\eta_{\underline{b}}(\{y_j\}) \leq \nu( \{y_j\})$, whilst simultaneously making
$\eta_{\underline{b}}(\{y_k\})$ smaller for each $k \in \{1, \ldots,n \} \setminus \{j \}$, thus contradicting the fact that $\underline{b}$ is a minimal element. Hence $\eta_{\underline{b}}(\{y_k\}) = \nu( \{y_k\})$ for all $k \in \{1, \ldots,n \}$. The fact that $\eta_{\underline{b}}$ and $\nu$ are equal follows from the fact that they are both probability measures with mean $y$.
Finally, let 
\[ B_\nu = ([0,\infty] \times [-\infty, y_0]) \cup \left( \cup_{i: 1 \leq i \leq n} \{ (s,y_i) ; s \geq b_i \} \right) \cup ([0,\infty] \times [y_{n+1} , \infty]). \]
Then $\tau_{B_{\nu}} \wedge \sigma \leq H_{y_0, y_{n+1}}$ and the result follows.

Now consider the general case of a measure $\nu$ on $[L,R]$ with mean $y$.
Let
\[ C_n = \{ k/n ; k = 0, \pm 1, \pm 2, \ldots, L < k/n < R \}  \cup \{L,R \}\]
and let $\sigma_n = \inf \{ t \geq \sigma :Y_t \in C_n \}$ and $\nu_n = \sL(Y_{\sigma_n})$. Then $\sigma_n$ is a stopping time and
$\nu_n$ has mean $y$ and finite support. By the study of the previous case there is a barrier $B_n$ such that $Y_{\tau_{B_n} \wedge \sigma_n}$ has law $\nu_n$ and $\tau_{B_n} \leq
H_{L,R}$. We want to show that down a subsequence $(B_n)_{n \geq 1}$ converges to a barrier $B$, $\tau_{B_n}$ converges almost surely to $\tau_B \leq H_{L,R}$ and $Y_{\sigma \wedge \tau_B} \sim \nu$.

By the compactness of $\sB_{[L,R]}$, $(B_n)_{n \geq 1}$ has a convergent subsequence. Let
$B$ be the limit. Moving to the subsequence, we may assume that $B_n \rightarrow
B$. Write $\tau_n$ as shorthand for $\tau_{B_n}$.

Note that $\E[H_{L,R}]$ is finite and choose $T > 2 \E[H_{L,R}] /\epsilon$; then
\[ \Prob( \tau_n \wedge \tau_B > T ) \leq
\frac{\E[\tau_n]}{T}  \leq \frac{\E[ H_{L,R} ]}{T}
< \frac{\epsilon}{2}  .\]
Fix $c > 0$. Choose
$\gamma>0$ such that
\[ \sup_{x \in [L,R]} \Prob^x \left[ \left( \sup_{\gamma < t < c} Y_t - x > \gamma \right)
\cap \left( \inf_{\gamma < t < c} Y_t - x < - \gamma  \right) \right] > 1 - \epsilon/2
\]
and $n_0$ such that
\[ \max \left\{ \sup_{(t,x) \in C_{n_0} }
\rho((t,x), B) , \sup_{(t,x) \in C}
\rho((t,x),\cup_{n \geq n_0} B_n ) \right\} < \gamma , \]
where $C_{n_0} = ([0,T] \times [L,R]) \cap ( \cup_{n \geq n_0} B_n)$ and
$C = ([0,T] \times [L,R]) \cap B$.
Then
\begin{equation*}
(|\tau_B - \tau_n| > c) \subseteq  (\tau_n \wedge \tau_B > T) \cup_{  (\tau_n \wedge \tau_B = t, Y_{\tau_n \wedge \tau_B} = x) \in [0,T] \times [L,R]} F(t,x)
\end{equation*}
where $F(s,y)$ is the set
\[ F(s,y) =\left( \left. \sup_{s+\gamma < t < s+c} Y_t - y \leq \gamma \right| Y_s = y\right)
\cup \left( \left. \inf_{s+ \gamma < t < s+ c} Y_t - y \geq - \gamma \right| Y_s = y\right).
\]

Clearly $\Prob(F(y,s)) \leq \epsilon/2$ for all $(y,s)$.
Hence by the Strong Markov property
\[ \Prob(|\tau_B - \tau_n| > c) < \epsilon, \]
and down a further subsequence if necessary, $\tau_n \rightarrow \tau_B$ almost surely.
Thus
\[ \sL(Y_{\sigma \wedge \tau_B}) = \lim_{n} \sL(Y_{\sigma_n \wedge \tau_n}) = \lim \nu_n = \nu.  \]

Also $\sigma \wedge \tau_B = \lim \sigma_n \wedge \tau_n \leq H_{L,R}$ so that $\sigma \wedge \tau_B$ is minimal and $\E[\sigma \wedge \tau_B] =\int q(x) \nu(dx)$.
\end{proof}

For a diffusion $Y$ with state space $I$, speed measure $m$ and initial value $Y_0=y$, and for a law $\nu$ on $[L,R]$ with mean $y$, we have that
  $E_Y(y;\nu) =  \int q_y(x) \nu(dx)$.
Clearly $E_Y(y;\nu) < \infty$ under the present conditions on $\nu$.

\begin{corollary}
  Suppose $\nu$ has mean $y$ and support in $[L,R] \subset (\ell,r)$. Then an embedding $\sigma$ of $\nu$ is minimal if and only if $\E[\sigma] = E_Y(y;\nu)$.
\end{corollary}
\begin{proof}
  By the first case of Theorem~\ref{thm:existence} there exists an embedding $\sigma$ of $\nu$ in $Y$, and then by Lemma~\ref{lem:m4} there exists a minimal embedding $\tilde{\sigma}= \sigma \wedge \tau_B$ with $\E[\tilde{\sigma}] = E_Y(y;\nu)$. If $\sigma$ is minimal then $\sigma = \tilde{\sigma}$ and
$\E[\sigma] = E_Y(y;\nu)$. Conversely, by the arguments at the end of Lemma~\ref{lem:minimalLR}, for any embedding $\E[\sigma] \geq E_Y(y;\nu)$ and so if
  $\E[\sigma] = E_Y(y;\nu)$ then $\sigma$ is minimal.
\end{proof}

\subsection{The general centred case}
Now suppose that $\nu$ is centred but that there is no subset $[L,R] \subset (\ell,r)$ for which $\nu([L,R])=1$.
We construct a sequence of measures $(\nu_n)_{n \geq n_0}$ with supports in bounded intervals $[L_n,R_n] \subset (\ell,r)$ and such that $(\nu_n)_{n \geq n_0}$ converges to $\nu$. Hence, given $\sigma$ and $\nu_n$ there
is a barrier $B_n$ with associated stopping time $\tilde{\sigma}_n = \tau_{B_n} \wedge \sigma$ such that $Y_{\tilde{\sigma}_n}$ has law $\nu_n$. For our specific choice of approximating sequence of measures we argue that the sequence of stopping times $\tau_{B_n}$ is monotonic increasing with limit $\tau_\infty$. Finally we show that $\sigma \wedge \tau_\infty$ is minimal and embeds $\nu$.

Recall that our current hypothesis is that $\nu$ is a measure on $\overline{I}$ such that $\nu \in L^1$ and $Y_0 = y = \overline{\nu}$. 

For a measure $\eta \in L^1$ with mean $c$ and support in $[\ell,r]$ define the 
potential $U_\eta : [\ell,r] \mapsto \R_+$ via $U_\eta(x) = \E^{Z \sim 
\eta}[|Z-x|]$. Let $\sV_{c}$ be the set of convex functions $f:[\ell,r] \mapsto 
\R$ satisfying $f(x) \geq |x-c|$, together with $\lim_{x \downarrow \ell} \{ 
f(x) - (c - x) \} = 0 = \lim_{x \uparrow r} \{ f(x) - (x-c) \}$. Then $U_\eta 
\in \sV_c$ and there is a one-to-one correspondence between elements of $\sV_c$ 
and probability measures on $[\ell,r]$ with mean $c$. For a pair of probability 
measures $\eta_i$ with support in $[\ell,r]$ we have that $\eta_1 \leq_{cx} 
\eta_2$ if and only if $U_{\eta_1}(x) \leq U_{\eta_2}(x)$ for all $x \in 
[\ell,r]$.

Given $\nu$, fix $n_0 \geq 1/U_\nu(\overline{\nu})$. For $n \geq n_0$ define $U_n : [\ell,r] \mapsto \R_+$ via
\[ U_n(x) = \max \{ U_\nu(x) - 1/n, |x- \overline{\nu}| \}, \]
and let $\nu_n$ be the probability measure with potential $U_n$. Then there exist $\{a_n,b_n\}$ such that
$[a_n,b_n] \subset (\ell,r)$, $\nu_n(A) = \nu(A)$ for all measurable subsets $A \subset (a_n,b_n)$ and
$\nu_n([\ell,a_n)) = 0 = \nu_n((b_n,r])$. Then $\nu_n$ has atoms at $a_n$ and $b_n$ and mean $\overline{\nu}$.
Further $(a_n)_{n\geq n_0}$ and $(b_n)_{n \geq n_0}$ are monotonic sequences and the family $(\nu_n)_{n \geq n_0}$ is increasing in convex order.

\begin{theorem}\label{thm:embeddingcentred}
Suppose $\nu \in L^1$ and $Y_0 = y =\overline{\nu}$.
Let $\sigma$ be an embedding of $\nu$. There exists an barrier $B$ such that $\tau_B \wedge \sigma$ also has law $\nu$ and $\E[\tau_B \wedge \sigma] = E_Y(y;\nu)$ where $E_Y(y;\nu)=\int q(y)\nu(dy)$.
\end{theorem}

\begin{proof}
For each $n$, fix $\nu_n$ as above. From our study of the bounded case we know there is a
barrier $B_n$ which we can assume contains $\{ (t,x), x \leq a_n \mbox{ or } x
\geq b_n \}$ such that $Y_{\tau_{B_n} \wedge \sigma}$ has law $\nu_n$.

We now show that if $p>n$ then $B_p \subset B_n$.

Let $\sB_n = \{ B \in \sB; \{(t,x): x \leq a_n \mbox{ or } x \geq b_n \}
\subseteq B, \sL(Y_{\tau_B \wedge \sigma}) \sim \nu_n \}$. We show that if
$n<p$, $B_n \in \sB_n$ and $B_p \in \sB_p$ then $B_n \cup B_p \in \sB_n$.
Certainly $\{(t,x): x \leq a_n \mbox{ or } x \geq b_n \} \subseteq B_n \cup
B_p$. Let $A_{n,p} = \{x: \inf \{ t : (t,x) \in B_n \} \leq \inf \{ t : (t,x)
\in B_p \} \}$.

Suppose $A \subset [a_n,b_n]$.
If $A \subset A_{n,p}$ and $Y_{\sigma \wedge \tau_{B_n \cup
B_p}}
\in A$ then $Y_{\sigma \wedge \tau_{B_n}} \in A$ and hence we have
\[ \nu_n(A) = \Prob(Y_{\sigma \wedge \tau_{B_n}} \in A)
            \geq \Prob(Y_{\sigma \wedge \tau_{B_n \cup B_p}} \in A).
\]
Conversely, if $A \subset A_{n,p}^c$, 
\[ \nu_n(A) = \nu(A) = \nu_p(A) = \Prob(Y_{\sigma \wedge \tau_{B_p}} \in A)
            \geq \Prob(Y_{\sigma \wedge \tau_{B_n \cup B_p}} \in A). \]
Thus for every set $A \subset [a_n, b_n]$, $\nu_n(A)  = \Prob(Y_{\sigma \wedge
\tau_{B_n}} \in A) \geq \Prob(Y_{\sigma \wedge \tau_{B_n \cup B_p}} \in A)$.
Hence there must be equality throughout and $B_n \cup B_p \in \sB_n$.

Now fix a sequence $(B_n)_{n \geq 1}$ with $B_n \in \sB_n$. Let $\tilde{B}_n$
be the closure of
$\cup_{i=n}^{\infty} B_i$. We aim to show that $\tilde{B}_n \in \sB_n$.
For $k>n$ let
\[ B^k_n = \cup_{i = n}^k B_i. \]
By the arguments of the previous paragraphs $B^k_n \in \sB_n$.
Since the set of barriers is compact,
$B^{k}_n$
converges to $\tilde{B}_n$ as $k \uparrow \infty$ and $\tau_{B^{k}_n}
\downarrow
\tau_{\tilde{B}_n}$
(note that $\tau_{B^k_n} \leq T_{a_n,b_n} < \infty$).
Hence, since paths of $Y$ are continuous, $\nu_n =
\lim_k \sL(Y_{\sigma \wedge \tau_{B^{k}_n}})
= \sL(Y_{\sigma \wedge \tau_{\tilde{B}_n}})$ and $\tilde{B}_n \in \sB_n$.
It follows that for $p>n$,
$\tilde{B}_p \subset \tilde{B}_n$, and without loss of generality we
shall assume that $B_p \subset B_n$.

Define $B_\infty = \cap B_n$ and set $\tau_\infty = \tau_{B_\infty}$. Then $\tau_{B_n} \uparrow \tau_{\infty}$. Also
$\tau_{B_n} \wedge \sigma \uparrow \tau_{\infty} \wedge \sigma$ and
\[ \sL(Y_{\tau_{\infty} \wedge \sigma}) = \lim \sL(Y_{\tau_{B_n} \wedge \sigma})
= \lim \nu_n = \nu. \]

It only remains to prove that $\E[\sigma \wedge \tau_\infty] = E_Y(y;\nu)$. But
\begin{equation*}
  \E[\sigma \wedge \tau_\infty] = \lim \E[\sigma \wedge \tau_{B_n}] = \lim E_Y(y; \nu_n) = \lim \int q(z) \nu_n(dz) = \int q(z) \nu(dz) .
\end{equation*}

\end{proof}

\subsection{The uncentred case}    
Without loss of generality we may assume that the mean of $\nu$ satisfies $\overline{\nu} < y$. Then for there to be an embedding of $\nu$ we must have that $I$ is unbounded above.

Again we construct a sequence of measures $(\nu_n)_{n \geq n_0}$ with supports in bounded intervals $[L_n,R_n] \subset (\ell,r)$ and such that $(\nu_n)_{n \geq n_0}$ converges to $\nu$. 

Recall that $\nu$ is a measure on $\overline{I}$ such that $\nu \in L^1$.

Let $F_{\nu}$ be the distribution function of $\nu$ and $F^{-1}_\nu$ the inverse. In particular, if $U \sim U[0,1]$ then $F^{-1}_\nu(U)$ has law $\nu$.

Suppose $\ell > -\infty$.
Fix $n_0 > \max \{y, (y - \overline{\nu})^{-1} \}$ and for $n \geq n_0$ let $v_n = F_{\nu}(n-)$ and let $u_n$ solve $\int_{u_n}^{v_n} \max \{ F^{-1}_\nu(u) , ( \ell + 1/n ) \} du + n(u_n + 1 - v_n)=y$. 
Then $Z_n := F^{-1}_{\nu}(U) I_{ \{ u_n < U \leq v_n \} } + n I_{ \{ (U \leq u_n) \cup (U > v_n) \} }$ has mean $y$. Let $\nu_n$ be the law of $Z_n$.
Now set $b_n = n$ and $a_n = \max \{ F^{-1}_{\nu}(u_n), (\ell + 1/n) \}$. For $A \subseteq (a_n,b_n)$ we have $\nu_n(A) = \nu(A)$ and moreover $\nu_n([\ell,a_n)) = 0 = \nu_n((n,\infty])$. The measure $\nu_n$ has an atom at $n$ of size $u_n + (1-v_n)$ (and potentially an atom at $a_n$) and mean $\overline{\nu}$.
Further $(a_n)_{n\geq n_0}$ is a decreasing sequence and the family $(\nu_n)_{n\geq n_0}$ is increasing in convex order.

If $\ell = - \infty$ then we can construct $\nu_n$ using a similar but simpler 
argument which does not require moving mass from the interval $(\ell,\ell+1/n)$ 
to $\ell + 1/n$. Then $u_n$ solves $\int_{u_n}^{v_n} F_{\nu}^{-1}(u) du + n(u_n 
+1 - v_n) = y$ and $a_n = F_\nu^{-1}(u_n)$.

Recall the definition of $E_Y(y;\nu)$ in (\ref{eqn:Edefu}).
Since we are assuming that $\nu \in L^1$ and $\overline{\nu} < y$, and since
$\lim_{n \to \infty} \frac{ q_y( y + n) }{n} = m(y, \infty)$, this simplifies to
\begin{equation}{\label{eqn:Edefu2}}
 E_Y(y;\nu) = \int q_y(z) \nu(dz) + 2(y - \overline{\nu}) m(y,\infty)
\end{equation}

\begin{theorem}\label{thm:embeddingnoncentred} Suppose $\nu \in L^1$.
Let $\sigma$ be an embedding of $\nu$. There exists an barrier $B$ such that $\tau_B \wedge \sigma$ also has law $\nu$ and $\E[\tau_B \wedge \sigma] = E_Y(y;\nu)$.
\end{theorem}

\begin{proof} It only remains to cover the case where $Y_0 = y \neq \overline{\nu}$. We may assume $y > \overline{\nu}$.

For each $n$, fix $\nu_n$ as above. From our study of the bounded, centred case we know there is a
barrier $B_n$ which we can assume contains $\{ (t,x), x \leq a_n \mbox{ or } x
\geq b_n \equiv n \}$ such that $Y_{\tau_{B_n} \wedge \sigma}$ has law $\nu_n$.
Moreover, exactly as in the proof of Theorem~\ref{thm:embeddingcentred}, and with similar notation, it follows that if $p>n$ then $B_p \subset B_n$, that $\tau_{B_n} \uparrow \tau_\infty$ and that $\tau_\infty \wedge \sigma$ embeds $\nu$.

Finally we show that $\E[\sigma \wedge \tau_\infty] = E_Y(y;\nu)$.

Observe that $q$ is convex and so $\lim_n q(n)/n$ exists in $(0,\infty]$. Further
\[ y = \int x \nu_n(dx) = \int_{u_n}^{v_n} \max \{ F^{-1}_\nu(u), (\ell + 1/n) \} du + n (1 + u_n - v_n) \]
and hence $\lim_n n(1 + u_n - v_n)$ exists and is equal to $y - \overline{\nu}$. Then, as before
\begin{equation*}
  \E[\sigma \wedge \tau_\infty] = \lim \E[\sigma \wedge \tau_{B_n}] = 
\lim E_Y(y;\nu_n) = \lim \int q(z) \nu_n(dz)
\end{equation*}
but in this case
\begin{eqnarray*}
   \int q(x) \nu_n (dx)  &=& \int_{u_n}^{v_n} q(\max \{ F_{\nu}^{-1}(u), (\ell +1/n) \}) du + q(n) ( 1 - u_n - v_n) \\
      & \rightarrow& \int_{0}^{1} q(F_{\nu}^{-1}(u)) du + \lim_n \left\{ \frac{q(n)}{n} n( 1 + u_n - v_n) \right\} \\
      &=& \int q(x) \nu(dx) + (y - \overline{\nu}) \lim_n \left\{ \frac{q(n)}{n} \right\} \\
      & = & E_Y(y; \nu).
\end{eqnarray*}

\end{proof}

\begin{proof}[Proof of Theorem~\ref{thm:martingale} in the case $\nu \in L^1$]
If $E_Y(y;\nu)=\infty$ then since
any embedding has $\E[\sigma] \geq \E[\sigma \wedge \tau_\infty] = E_Y(y;\nu)$ there are no integrable embeddings.
Conversely, if $E_Y(y;\nu)<\infty$, then by Theorem~\ref{thm:embeddingcentred} or Theorem~\ref{thm:embeddingnoncentred} there exists an embedding $\tilde{\sigma}$ with $\E[\tilde{\sigma}] = E_Y(y; \nu)$.

Now suppose $E_Y(y;\nu)<\infty$ and $\sigma$ is an embedding of $\nu$.

Suppose $\sigma$ is minimal.
Choose $\nu_n$ as in the discussion before Theorem~\ref{thm:embeddingcentred} or 
Theorem~\ref{thm:embeddingnoncentred} as appropriate. In both of these 
theorems it was shown that we could choose a sequence of barriers
$B_n$
such that $\tau_{B_n} \wedge \sigma \rightarrow \tau_{B_\infty} \wedge \sigma$ and
$\tau_{B_\infty} \wedge \sigma$ embeds $\nu$. By minimality of $\sigma$,
$\tau_{B_\infty} \wedge
\sigma = \sigma$.
Then, since $\tau_{B_n} \wedge \sigma$ is
increasing,
\[ \E[\sigma] = \E[\tau_{B_\infty} \wedge \sigma] = \lim_n \E[ \tau_{B_n} \wedge \sigma ]
= \lim_n \int q(x) \nu_n (dx) = E_Y(y; \nu). \]

Conversely, if $\sigma$ is not minimal then there is an embedding $\hat{\sigma}$ of $\mu$ with $\hat{\sigma} \leq \sigma$, $\Prob(\hat{\sigma} < \sigma)>0$ and $\hat \sigma$ integrable. Then $\E[\sigma] > \E[\hat{\sigma}] \geq E(y;\nu)$.
\end{proof}

\begin{example}
The following example shows that unlike in the Brownian case, in general integrability alone is not sufficient for minimality.

Suppose the diffusion $Y$ solves $dY_t = (1+Y_t^2) dW_t$ subject to $Y_0=0$. Let 
$\nu = \frac{1}{2} \delta_{1} + \frac{1}{2} \delta_{-1}$ so that $\nu$ is 
uniform measure on $\{ \pm 1 \}$. Let $\hat{H} = H^Y_{1} \vee H^Y_{-1}$.
Then $\hat{H}$ embeds $\nu$ and $\E[\hat{H}] < \infty$, but 
$\hat{H}$ is not minimal since $\hat{H} > H^Y_{1} \wedge H^Y_{-1}$ which is also 
an 
embedding of $\nu$.

\end{example}

\begin{example}
This example gives another circumstance in which integrability is not sufficient to guarantee minimality.

Let $Y$ be a time-homogeneous martingale diffusion on $I=[\ell,r]$ with $-\infty < \ell < y < r < \infty$.
Suppose $\ell$ and $r$ are exit boundaries and that $\E[H^Y_{\ell,r}] < \infty$. We take $\ell$ and $r$ to be absorbing boundaries. (A simple example is obtained by taking Brownian motion started at $y$ and absorbed at $\ell$ and $r$.) Let $\nu = (r-y)/(r-\ell) \delta_\ell + (y - \ell)/(r-\ell) \delta_r$. Then for $c>0$, $H^Y_{\ell,r} + c$ is an integrable embedding which is not minimal.

However, examples of this type are degenerate and may easily be excluded by restricting the class of embeddings to those satisfying $\sigma \leq H^Y_{\ell,r}$.
\end{example}

\begin{example}
Now we give an example which shows that minimality alone is not sufficient for integrability.

Let $Y$ be geometric Brownian motion so that $Y$ solves $dY_t = Y_t dW_t$. Let $Y$ have initial value $Y_0 = 1$. It is easy to see that for $a \in (0,1]$ we have
\[
\E[H_a] = 2 \int_a^\infty [(z \wedge 1) - a] \frac{dz}{z^2} = 2 \log\left( 
\frac{1}{a} \right).
\]
Let $\nu = \delta_0$. Then $\tau = \infty$ is the minimal stopping time that embeds $\nu$ in $Y$. Obviously $\tau$ is not integrable.

More generally, let $\nu$ be any probability measure on $(0,1)$ with $\int \log 
y \; \nu(dy) = - \infty$, and let $Z$ be a random variable such that $\sL(Z) 
\sim \nu$. Let the filtration ${\mathbb F} = (\sF_t)_{t \geq 0}$ be such that 
$Z$ is $\sF_0$-measurable, and let $W$ be a ${\mathbb F}$-Brownian motion which 
is independent of $Z$.

Let $\tau = \inf \{ u \geq 0 : Y_u = Z \}$. Then $\tau$ is an embedding of 
$\nu$. Note that $\tau$ is a stopping time with respect to ${\mathbb F}$ but not 
with respect to the smaller filtration generated by $Y$ alone. Moreover, 
\[ \E[\tau] = -2 \int \log z \; \nu(dz) = \infty \]

Observe that $q_1(x) = \int_1^x \int_1^y \frac{2}{z^2} dz dy = 2(x-1) - 2 \log(x)$, and hence $\lim_{x \to \infty} q_1(x)/x = 2$. Therefore, for any law $\nu$ on $(0,1)$, for a minimal embedding
\[ \E[\tau] =  2 \int (x-1) \nu(dz) - 2 \int \log z \; \nu(dz) + 2(1 - \overline{\nu}) =  - 2 \int \log z \; \nu(dz). \]

We give another example of a minimal non-integrable embedding which does not require independent randomisation in the section on the Az\'ema-Yor stopping time.

Another feature of this example, is that $Y$ is a martingale and yet it is easy to construct examples with $\overline{\nu}<y$ for which there is an integrable embedding. Hence integrability and minimality of $\tau$ is not sufficient for uniform integrability of $(Y_{t \wedge \tau})_{t \geq 0}$.
\end{example}



\section{Alternative characterisations of $E$}
\label{sec:alternative}
In the comments before Theorem~\ref{thm:martingale} we argued that in the non-centred case a natural family of embeddings was those which first involved waiting for the process to hit $\overline{\nu}$ and then to embed $\nu$ in $Y$ started at $\overline{\nu}$. For a stopping rule $\tau$ as given in (\ref{eqn:taudef}) we have from the analysis of the centred case that
\begin{equation}
\label{eqn:altE1}
 \E[\tau] = \E^y[H_{\overline{\nu}}] + E_Y(\overline{\nu};\nu)
 \end{equation}
 Now we want to show that the right hand side of (\ref{eqn:altE1}) is equivalent to the expression given in (\ref{eqn:Edefu}).

 More generally, for $v \in [\overline{\nu},y]$ we could imagine waiting for the process to hit $v$ and then using a minimal embedding time to embed $\nu$ in $Y$ started at $v$. Then we find
\begin{equation}
\label{eqn:altE2}
 \E[\tau] = \E^y[H_v] + E_{Y}(v;\nu)
 \end{equation}
 We want to show that the right-hand-side of (\ref{eqn:altE2}) does not depend on $v$.

 \begin{lemma}
For $v \in [\overline{\nu},y]$,
 \[ G(v) = 2\int_v^\infty ( y \wedge z - v ) m(dz) + \int q_{v}(z) \nu(dz) + (v - \overline{\nu}) \lim_{n \uparrow \infty} \frac{q_v(v+n)}{n} \]
 does not depend on $v$. In particular, for all $v \in [\overline{\nu},y]$,
 $E_{Y}(y, \nu) = \E^y[H_v] + E_{Y}(v;\nu)$. If this expression is finite for any (and then all) $v \in [\overline{\nu},y]$ we have that $\E[\tau] = \E^y[H_v] + E_Y(v;\nu)$.
 \end{lemma}

 \begin{proof}
 For any $u,v$,
 \[ q_u(z) = q_u(v) + q_v(z) + q'_u(v)(z-v) . \]
 Then, with $u = {\overline{\nu}}$, $q_v(z) = q_{\overline{\nu}}(z) - q_{\overline{\nu}}(v) + q'_{\overline{\nu}}(v)(v-z)$ and 
\begin{eqnarray*}
  G(v) &=& 2\int_v^y (z - v ) m(dz) + 2(y-v) \int_y^\infty m(dz) + \int q_{\overline{\nu}}(z) \nu(dz) - q_{\overline{\nu}}(v) \\
       & & \hspace{5mm} + (v - \overline{\nu})q'_{\overline{\nu}}(v) + 2(v - \overline{\nu}) \int_v^\infty m(dz)  \\
       &=& 2(y- \overline{\nu}) \int_y^\infty m(dz) + \int q_{\overline{\nu}}(z) \nu(dz) + 2\int_{\overline{\nu}}^y (z - \overline{\nu}) m(dz)
\end{eqnarray*}
which does not depend on $v$.
\end{proof}

\section{Extensions}

\subsection{Non-integrable target laws}
We have seen that if $\nu \in L^1$ then there exists an integrable embedding of $\nu$ if $\E^y[H_{\overline{\nu}}]$ and $\int q_{\overline{\nu}}(x) \nu(dx)$ are both finite. In this short section we argue that if $Y_0= y \in (\ell,r)$ and $\nu \notin L^1$ then there does not exist an integrable embedding of $\nu$.

Note first that $q=q_y$ is non-negative and convex, and hence $q(x) \geq \alpha |x-y| - \beta$ for some pair of finite positive constants $\alpha, \beta$. Let $T_n$ be a localising sequence for the local martingale $\{ q(Y_{t \wedge \sigma}) - (t \wedge \sigma) \}_{t \geq 0}$. Then, by an argument similar to that in the proof of Lemma~\ref{lem:minimalLR}
\[ \E[\sigma] = \lim_n \E[ \sigma \wedge T_n] = \liminf \E[q(Y_{\sigma \wedge T_n})] \geq \E[ \liminf q(Y_{\sigma \wedge T_n})] = \int q(z) \nu(dz) = \infty. \]

\subsection{Diffusions started at entrance points}
\label{ssec:notL1}

In the proofs of the main results we assumed that $Y$ started at an interior point in $(\ell,r)$. Now we consider what happens if we start at a boundary point. The motivating example is a Bessel process in dimension 3 started at zero.

 After a change of scale we may assume that we are working with a diffusion in natural scale. Then, if the boundary point is finite and an entrance point, it must also be an exit point (for terminology, see Borodin and Salminen~\cite[Section II.6]{BorodinSalminen:02}).
 We have assumed exit boundary points to be absorbing. It follows that an entrance point must be infinite; without of generality we assume that $Y$ starts at $+\infty$ and that $I = (\ell,\infty)$ where we may have $\ell = -\infty$.

So suppose that $\infty$ is an entrance-not-exit point.  In particular, 
$\E^{\infty}[H_z]<\infty$ for some $z \in (\ell, \infty)$ or equivalently 
$\int^\infty z m(dz) < \infty$. We show that the results of previous sections 
pass over to this case with a small modification.

We suppose the initial sigma algebra $\sF_0$ is sufficiently rich as to include an independent, uniformly distributed random variable.

\begin{theorem}
\label{thm:entrance}
Suppose $Y$ is a diffusion in natural scale on $I = (\ell,\infty)$ and suppose $Y_0=\infty$, where $\infty$ is an entrance point.
Then there exists an integrable embedding of $\nu$ if and only if $E_Y(\infty;\nu)$ defined by
\begin{equation}
\label{eqn:Enatural}
 E_Y(\infty;\nu):= 2 \int^\infty_\ell \nu(dx) \int_x^\infty m(dz) (z-x)
 \end{equation}
is finite.
Furthermore,
if there exists an integrable embedding, then every minimal embedding $\sigma$ has $\E[\sigma] =E_Y(\infty;\nu)$.

\end{theorem}

\begin{remark}
\label{rem:entrance}
Note that $E_Y(\infty;\nu)$ can be rewritten as
\[ E_Y(\infty;\nu) = 2 \int^\infty_\ell m(dz) \int_\ell^z \nu(dx) (z-x) \]
It follows that if $\ell = - \infty$ and $\int_{-\infty}^{0} |x| \nu(dx) = \infty$ then $E_Y(\infty;\nu) = \infty$.

However, if $\nu$ has support in $[L,\infty]$ for $L>\ell$ or if $\int_{\ell} m(dz)$ and $\int^0_\ell |x| \nu(dx)$ are finite (the latter is always true if $\ell > -\infty$),
then it is possible to have $\nu \notin L^1$ and still have $E_Y(\infty;\nu)<\infty$ and the existence of integrable embeddings. For example, suppose $Y$ solves $dY_t = Y_t^2 dB_t$ subject to $Y_0 = \infty$ and suppose $\nu$ is a measure on $(0,\infty)$ with $\int_0^\infty x \nu(dx) = \infty$ and $\int_0^\infty \nu(dx) / x^2 < \infty$, eg $\nu([x,\infty))= x^{-1} \wedge 1$.
Then $E_Y(\infty;\nu) = \int x^{-2} \nu(dx)/3 < \infty$ but $\nu \notin L^1$.

Suppose instead that $\nu \in L^1$. Then as in Section~\ref{sec:alternative} we can rewrite $E_Y(\infty;\nu)$ as
\[ E_Y(\infty;\nu) = 2 \int_{\overline{\nu}}^\infty (y-\overline{\nu}) m(dy) + \int q_{\overline{\nu}}(y) \nu(dy) \]
This last expression as a clear interpretation as the sum of $\E^\infty[H_{\overline{\nu}}]$ and the expected time to embed law $\nu$ in $Y$ started at $\overline{\nu}$ using a minimal embedding. It follows that if $\nu \in L^1$ and there exists an integrable embedding of $\nu$ started at $\overline{\nu}$ then the stopping time `run until $Y$ hits the mean, and then use a minimal embedding to embed $\nu$ in $Y$ started from the mean' is a minimal and integrable embedding.
\end{remark}

\begin{proof}[Proof of Theorem~\ref{thm:entrance}]
Suppose first that $E_Y(\infty;\nu)$ is finite.
By assumption $\sF_0$ is sufficiently rich as to include a uniform random variable. (Note that if $\nu$ includes an atom at $\infty$ independent randomisation of this form will always be necessary to construct an embedding.) Then there exists a random variable $Z$ with law $\nu$ and setting $\sigma = \inf \{ u \geq 0; Y_u \leq Z \}$ we have $Y_\sigma \sim \nu$ and
\[ \E[\sigma] = \int \nu(dz) \E^{\infty}[H^Y_z] = 2 \int \nu(dz) \int_z^\infty (y-z) m(dy) = E_Y(\infty;\nu) . \]

If $\nu \in L^1$ then we do not need independent randomisation. In this case both $\E^{\infty}[H_{\overline{\nu}}]$ and $\int q_{\overline{\nu}}(y) \nu(dy)$ are finite (since $E_Y(\infty;\nu)$ is). Then there exists a minimal and integrable embedding $\tau^{\overline{\nu},\nu}$ of $\nu$ in $Y$ started at $\overline{\nu}$
and
\[ \tau = H_{\overline{\nu}} + \tau^{\overline{\nu},\nu} \circ \Theta_{H_{\overline{\nu}}} \]
is an integrable embedding.

Now suppose there is an integrable embedding. Then there exists an integrable minimal embedding $\sigma$ say. The remaining parts of the theorem will follow if we can show that $\E[\sigma] = E_Y(\infty;\nu)$.

So, suppose $\sigma$ is integrable and minimal.  Since $\infty$ is an entrance boundary, there exists $N$ such that $\E^{\infty}[H_N]< \infty$.

For $n \geq N$ 
let $\tilde{\sigma}_n = \max \{ \sigma, H_{n} \}$ and let $\nu_n = \sL(Y_{\tilde{\sigma}_n})$. Write $\tilde{\sigma}_n = H_{n} + \hat{\sigma}_n$ where $\hat{\sigma}_n = (\sigma - H_{n})^+$ and let $\hat{\nu}_n = \sL(Y^{n}_{\hat{\sigma}_n})$ where here the superscript reflects the fact that $Y$ starts at $n$.

First we argue that for each $n \geq N$, $\hat{\sigma}_n$ is minimal for $\hat{\nu}_n$ in $Y$ started at $n$. Suppose $\hat{\rho}_n \leq \hat{\sigma}_n$ also embeds $\hat{\nu}_n$ in $Y$ started from $n$. If $\rho$ is defined by
\[ \rho = \left\{ \begin{array}{ll} \sigma & \sigma < H_{n} \\
                                     H_{n} + \hat{\rho}_n & \sigma \geq H_n
                                     \end{array} \right. \]
then $\rho \leq \sigma$ and $Y_\rho \sim Y_\sigma$. By minimality of $\sigma$ we conclude that $\rho = \sigma$ and hence $\hat{\rho}_n = \hat{\sigma}_n$.

Since $\hat{\sigma}_n$ is minimal (and integrable, since $\sigma$ is integrable and $\E^{\infty}[H_n] \leq  \E^{\infty}[H_N] < \infty$) we have that
$\E^n [ \hat{\sigma}_n ] = E_Y(n,\hat{\nu}_n) = \int q_n(x) \tilde{\nu}_n(dx) + 2(n- \overline{\tilde{\nu}_n}) m((n,\infty))$.
Then
\begin{eqnarray}
\E^\infty [\sigma] & = & \lim_n \E[ (\sigma - H_n)^+] \nonumber \\
                  & = & 2 \lim_n \left\{ \int_{\ell}^{\infty} \tilde{\nu}_n(dx) \int_n^x m(dz) (x-z) + (n- \overline{\tilde{\nu}_n}) m((n,\infty)) \right\} \label{eqn:entrance}
\end{eqnarray}
Since $\infty$ is an entrance boundary $\int^{\infty} y m(dy) < \infty$ and hence $\lim_n n m((n,\infty) = 0$.
Further, since $\tilde{\nu}_n = \nu$ on $(-\infty,0)$, $\int_{-\infty}^0 |x| \tilde{\nu}_n(dx) < \infty$ if and only if  $\int_{-\infty}^0 |x| \nu(dx) < \infty$. But, if $\int_{-\infty}^0 |x| \nu(dx) = \infty$, then for any embedding $\rho$ of $\nu$ in $Y$ started at $\infty$ we have
\[ \E^\infty[\rho] > \E^\infty[(\rho - H_0)^+] > \int_{-\infty}^0 \nu(dx) q_0(x) = \infty \]
and hence there cannot be an integrable embedding of $\nu$. Since such an embedding exists by hypothesis, we must have $\int_{-\infty}^0 |x| \nu(dx)<\infty$. Then $\tilde{\nu}_n \in L^1$ and $\overline{\tilde{\nu}_n} \uparrow \overline{\nu} \in (-\infty,\infty]$. In particular, $\lim_n (n- \overline{\tilde{\nu}_n})m(n,\infty) \rightarrow 0$.

For the first term in (\ref{eqn:entrance}), since $\tilde{\nu}_n = \nu$ on $(\ell, n)$,
\begin{eqnarray*}
2\lim_n \left\{ \int_{\ell}^{\infty} \tilde{\nu}_n(dx) \int_n^x m(dz) (x-z) \right\} & \geq &
2\lim_n \left\{ \int_{\ell}^{n} \nu(dx) \int^n_x m(dz) (z-x) \right\} \\
 & = & 2\int_{\ell}^{\infty} \nu(dx) \int^\infty_x m(dz) (z-x),
\end{eqnarray*}
and conversely, since $\tilde{\nu}_n \leq \nu$ on $(n,\infty)$,
\begin{eqnarray*}
2\lim_n \left\{ \int_{\ell}^{\infty} \tilde{\nu}_n(dx) \int_n^x m(dz) (x-z) \right\}& \leq & 2 \lim_n \left\{  \int_{\ell}^{\infty} \nu(dx) \int_n^x m(dz) (x-z) \right\} \\
& = & 2 \int_{\ell}^{\infty} \nu(dx) \int^\infty_x m(dz) (z-x) .  
\end{eqnarray*}

\end{proof}

\section{Recovering results for general diffusions}

Let $X = (X_t)_{t \geq 0}$ be a time-homogeneous one-dimensional diffusion with state space $I^X$ and suppose $X$ solves $dX_t = a(X_t)dW_t + b(X_t)dt$ subject to $X_0=x$. Then provided $b/a^2$ and $1/a^2$ are locally integrable, $X$ has scale function $s=s^X$ and speed measure $m^X$ given by
\[ s'(z) = \exp \left( - \int^z \frac{2 b(v)}{a(v)^2} dv \right), \hspace{10mm} m^X(dz) = \frac{dz}{a(z)^2 s'(z)} .\]
Now let $Y = (Y_t)_{t \geq 0}$ be given by $Y_t = s^X(X_t)$. Then $Y$ is a diffusion in natural scale with state space $I = s^X(I^X)$ and speed measure
\[ m(dy) = m^X( d s^{-1}(y)) = \frac{dy}{a(s^{-1}(y))^2 s'(s^{-1}(y))^2} ,\]
so that for $[L,R] \subset I$, $m((L,R)) = m^X((s^{-1}(L), s^{-1}(R)))$.

Then $X_\tau \sim \mu$ is equivalent to $Y_\tau \sim \nu$ where $\nu(A) = \mu \circ s^{-1}(A)$.

We have that $\overline{\nu} := \int_I v \nu(dv) = \int_{I^X} s(z) \mu(dz)$ and $\int_I q_{s(x)}(v) \nu(dv) = \int_{I^X} q_{s(x)}(s(z)) \mu(dz)$. Moreover, $q_y(z) = 2 \int_y^z (z-w) m(dw) = 2 \int_{s^{-1}(y)}^{s^{-1}(z)} ( z - s(v) ) m^X(dv)$.

For definiteness suppose $s(x) \geq \overline{\nu}$, and denote by $r$ the upper limit of $I$ and by $r^X$ the upper limit of $I^X$. Then $r=\infty$ and
\begin{eqnarray*}
E_Y(s(x), \nu) & = & \int_I q_{s(x)}(z)\nu(dz) + 2(s(x) - \overline{\nu})m((s(x),r))    \\
& = & \int_{I^X} q_{s(x)}(s(z)) \mu(dz) + 2(s(x) - \overline{\nu})m^X((x,r^X))    \\
& = & 2 \int_{I^X} \left\{ \int_x^z (s(z)-s(v)) m^X(dv) \right\} \mu(dz) + 
2(s(x) - \overline{\nu})m^X(x,r^X)).
\end{eqnarray*}

In general therefore, for $x \in int(I^X)$ set
$E_X(x;\mu) = \infty$ if $\int_{I^X} |s(z)| \mu(dz) = \infty$ and otherwise
\begin{eqnarray}
\nonumber
E_X(x;\mu) & = &
2 \int_{I^X} \left\{ \int_x^z (s(z)-s(v)) m^X(dv) \right\} \mu(dz) \\
\label{eqn:EXdef} && \hspace{5mm} + 2|s(x) - \overline{\nu}| \left(m^X((x,r^X)) \sI_{ \{ s(x) > \overline{\nu} \} } +  m^X((l^X,x)) \sI_{ \{ s(x) < \overline{\nu} \} } \right) 
\end{eqnarray}
where $\sI$ is the indicator function.
As is the case for diffusions in natural scale, there is a second representation of $E_X$ in terms of the expected value of first hitting time of the weighted mean of the target law together with the expected value of an embedding in a process started at the weighted mean, namely
\begin{equation}
\label{eqn:EXdef2}
E_X(x;\mu) = \E^{x}[H^X_{s^{-1}(\overline{\nu})}] + \int q_{\overline{\nu}}(s(z)) \mu(dz).
\end{equation}
Note that in this expression $q$ is defined for the transformed process in natural scale.

\begin{proof}[Proof of Theorem~\ref{thm:general}]
$\tau$ is minimal for $\mu$ in $X$ started at $x$ if and only if $\tau$ is minimal for $\nu$ in $Y$ started at $y=s(x)$. Furthermore,
$\tau$ is an integrable embedding of $\mu$ if and only if $\tau$ is an integrable embedding of $\nu$. Then $\E[\tau]=E_Y(s(x);\nu) = E_X(x;\mu)$, where $E_X$ is defined in either (\ref{eqn:EXdef}) or (\ref{eqn:EXdef2}).
\end{proof}

\begin{example}
Suppose $P$ is a Bessel process of dimension 3, started at $p>0$. Then the scale function is $s(x)=-x^{-1}$ and $I=(-\infty,0)$. The speed measure is $m^P(dp) = p^2 dp$. There exists an embedding of $\mu$ in $Y$ if and only if $\overline{\nu} \geq - p^{-1}$ where $\overline{\nu} =  - \int_0^\infty x^{-1} \mu(dx)$.
Further, there exists an integrable embedding of $\mu$ if and only if $E_P(p;\mu)<\infty$ where
\begin{eqnarray*}
E_P(p;\mu) & = & \int_0^{\infty} \mu(dz) 2 \int_p^z \left( \frac{1}{v} - \frac{1}{z} \right)v^2 dv + 2 \left( \frac{1}{p} + \overline{\nu} \right) \frac{p^3}{3} \\
& = & \frac{1}{3} \int_{0}^{\infty} z^2 \mu(dz) - \frac{p^2}{3}
\end{eqnarray*}
\end{example}

\begin{example}
Suppose $X$ is given by $X_t=aW_t+ bt$ where $b>0$ and $W$ is standard Brownian motion, null at zero. Then $s(z) = - e^{-2bz/a^2}$ and $m^X(dz) = dx e^{2bz/a^2}/2b$. Set $\overline{\nu} = - \int_{\R} e^{-2bz/a^2} \mu(dz)$ and suppose $\overline{\nu} \in [-1,0]$, else there is no embedding.
Then $s^{-1}(\overline{\nu}) = - \frac{a^2}{2b} \log | \overline{\nu}|$ and $\exp( - \frac{2b}{a^2} s^{-1}(\overline{\nu})) = |\overline{\nu}|$. Hence
\begin{eqnarray*}
\int \mu(dz) q_{\overline{\nu}}(s(z)) & = &
\int \mu(dz) 2 \int_{s^{-1}(\overline{\nu})}^z (s(z) - s(v)) m^X(dv) \\
& = & \int \mu(dz) 2 \int_{s^{-1}(\overline{\nu})}^z (e^{-2bv/a^2} - e^{-2bz/a^2}) \frac{dv}{2b} e^{2bv/a^2} \\
& = & \int \mu(dz) 2 \int_{s^{-1}(\overline{\nu})}^z (1 - e^{2b(v-z)/a^2}) \frac{dv}{2b}  \\
& = & \int \mu(dz) \left\{ \frac{z}{b} - \frac{s^{-1}(\overline{\nu})}{b} - \frac{a^2}{2 b^2} + \frac{a^2}{2b^2} e^{2b(s^{-1}(\overline{\nu})-z)/a^2} \right\} \\
& = & \frac{1}{b} \int z \mu(dz) + \frac{a^2}{2b^2} \log |\overline{\nu}| - \frac{a^2}{2b^2} +\frac{a^2}{2b^2} \frac{1}{|\overline{\nu}|} \int e^{- 2bz/a^2} \mu(dz) \\
& = & \frac{1}{b} \int z \mu(dz) + \frac{a^2}{2b^2} \log |\overline{\nu}|.
\end{eqnarray*}

Suppose $X_0=x$. For $w>x$ we have $\E^x[H^X_w] = (w-x)/b$. Then, using (\ref{eqn:EXdef2}),
\[ E_X(x;\mu) = \E^x[H^X_{s^{-1}(\overline{\nu})}] + \int q_{\overline{\nu}}(s(z)) \mu(dz) = \frac{1}{b} \left( \int z \mu(dz) - x \right). \]
Recall from Proposition~\ref{prop:minimal} that every embedding of $\mu$ is minimal.
Then, for drifting Brownian motion, every embedding of $\mu$ has the same expected value.
\end{example}

\begin{remark}
Drifting Brownian motion was the subject of Grandits and Falkner~\cite{GranditsFalkner:00}, and the conclusion of the previous example is contained in their Proposition 2.2. Note that in the case $X_t = x + aB_t + bt$, if $\E[\tau]< \infty$ then $\E[X_\tau] - x = b \E[\tau]$. Hence, for an embedding $\tau$ of $\mu$ the result $\E[\tau]= E_X(x;\mu) = (\int z \mu(dz) - x)/b$ is not unexpected, and can be proved directly by other means.
\end{remark}

\section{Minimality and Integrability of the Az\'ema-Yor embedding}
Az\'ema and Yor~\cite{AzemaYor:79a,AzemaYor:79b} (see also Rogers and Williams~\cite[Theorem VI.51.6]{RogersWilliams:00b} and Revuz and Yor~\cite[Theorem VI.5.4]{RevuzYor:99}), give an explicit construction of a solution of the SEP for Brownian motion. The original paper~\cite{AzemaYor:79a} assumes the target law is centred and square integrable, but the $L^2$ condition is replaced with a uniform integrability condition in \cite{AzemaYor:79b}, see also \cite{RevuzYor:99}. Az\'ema and Yor~\cite{AzemaYor:79a} also indicate how the results can be extended to diffusions, provided that the process is recurrent and provided that once the process has been transformed into natural scale, the mean of the target law is equal to the initial value of the diffusion.

The Az\'ema-Yor stopping time for a centred target law $\nu$ in Brownian motion $W$ null at zero is
\begin{equation}
\label{tauAY}
 \tau^W_{AY,\nu} = \inf\{ u : W_u \leq \beta_{\nu} (J^W_u) \},
\end{equation}
where $J^W$ is the maximum process $J^W_u = \sup_{s \leq u} W_u$, and 
$\beta_{\nu}$ is the left-continuous inverse barycentre function, ie $\beta_\nu 
= b_{\nu}^{-1}$ where for a centred distribution $\eta$, $b_\eta(x) = \E^{Z \sim 
\eta}[ Z | Z \geq x]$. The Az\'ema-Yor embedding has become one of the canonical solutions of the SEP because it does not involve independent randomisation and because it is possible to give an explicit form for the stopping time. Further, amongst uniformly integrable (or equivalently minimal) solutions of the SEP for Brownian motion, the Az\'ema-Yor solution has the property that it maximises the law of the stopped maximum, ie for all increasing functions $H$, $\E[H(J^W_\tau)]$ is maximised over minimal embeddings $\tau$ of $\nu$ in $W$ by $\tau^W_{AY,\nu}$.

In the case where $\nu \in L^1$ but $\nu$ is not centred, Pedersen and Peskir~\cite{PedersenPeskir:01} make the simple observation that we can embed $\nu$ by first running the Brownian motion until it hits $\overline{\nu}$ and then embedding $\nu$ in Brownian motion started at $\overline{\nu}$ using the classical centred Az\'ema-Yor embedding, ie they propose
\[ \tau^W_{PP,\nu} = H^W_{\overline{\nu}} + \tau^W_{AY,\nu} \circ \Theta_{H^W_{\overline{\nu}}}. \]
However, if the Brownian motion is null at zero, and $\overline{\nu}<0$, then 
the embedding $\tau_{PP,\nu}$ no longer maximises the law of the stopped 
maximum. Instead Cox and Hobson~\cite{CoxHobson:04} introduce an alternative 
modificiation of the Az\'ema-Yor stopping time which does maximise the law of 
the stopped maximum, and it is this embedding which we will study here. In fact 
the expected value of any embedding of the form $H^Y_{\overline{\nu}} + 
\tau^{\overline{\nu},\nu} \circ \Theta_{H^Y_{\overline{\nu}}}$ can be found very 
easily, and our aim here is to analyse an embedding which is not of this form.

Suppose $W_0 = w$ and $\nu \in L^1$. Define $D_\nu(x) = \E^{Z \sim \nu}[(Z-x)^+] + (w - \overline{\nu})^+$ and for
$z \geq w$ set
\begin{equation}
\label{eqn:betadef}
 \beta_\nu(z) = \arg \inf_{v<z} \left\{ \frac{D_\nu(v)}{z-v} \right\}.
 \end{equation}
(Here the $\arg \inf$ may not be uniquely defined, but we can make the choice of $\beta_\nu$ unique by adding a left-continuity requirement.) 
Then the Cox-Hobson extension of the Az\'ema-Yor embedding is to set
\begin{equation}
\label{eqn:tauCH}
 \tau^W_{CH,\nu} = \inf\{ u : W_u \leq \beta_{\nu} (J^W_u) \} .
\end{equation}
Note that if $\overline{\nu} \geq w$, then for $z \in [w,\overline{\nu}]$ we have $\beta_\nu(z) = - \infty$. In this case the Cox-Hobson and Pedersen-Peskir embeddings are identical. However, if $\overline{\nu}<w$ then the Cox-Hobson and Pedersen-Peskir embeddings are distinct.

To ease the exposition we assume that $\nu$ has a density $\rho$. (The general case can be recovered by approximation, or by taking careful consideration of atoms.) Then $b = \beta_{\nu}^{-1}$ solves
\begin{equation}
\label{eqn:defbarycentre} 
(b(y) - y) \nu((y,\infty)) = D_{\nu}(y), 
\end{equation}
$b$ is differentiable and $\nu((y,\infty))b'(y) = (b(y)-y) \rho(y)$.
Then, writing $\tau$ for $\tau^W_{CH,\nu}$ and $L(\nu)$ for the lower limit of the support of $\nu$ and using excursion-theoretic arguments,
\begin{eqnarray*}
\Prob(W_\tau > y)  = \Prob(J^W_\tau > b(y)) & = & \exp \left( - \int_w^{b(y)} 
\frac{dz}{z - \beta(z)} \right) \\
& = & \exp \left( - \int_{w \vee {\overline{\nu}}}^{b(y)} \frac{dz}{z - \beta(z)} \right) \\
& = & \exp \left( - \int_{L(\nu)}^{y} \frac{b'(v)}{b(v)  - v} dv \right) \\
& = & \exp \left( - \int_{L(\nu)}^{y} \frac{\rho(v)}{\nu((v,\infty))} dv \right) = \nu((y,\infty))
\end{eqnarray*}
and hence $\tau^W_{CH,\nu}$ is an embedding of $\nu$.

Cox and Hobson~\cite{CoxHobson:06} prove that the embedding in (\ref{eqn:tauCH}) is minimal. A bi-product of the subsequent arguments in this section is a proof of minimality by different means. Note that this is only relevant in the case $I=\R$, else every embedding is minimal.

Let $Y$ be a regular diffusion in natural scale. Then by the Dambis-Dubins-Schwarz theorem $Y$ can be written as a time-change of Brownian motion: $Y_t = W_{[Y]_t}$ for some Brownian motion (on a filtration and probability space constructed from the original space supporting $Y$). Then if we set $Q = [Y]^{-1}$ we have $W_t = Y_{Q_t}$. Conversely, let $W$ be Brownian motion and let $(L^W_t(z))_{t\geq 0,z\in \R}$ be its family of local times. Given a measure $m$  on $I$ (with a strictly positive density with respect to Lebesgue measure), set $A_s = \int_I m(dz) L^W_s(z)$. Then $A$ is strictly increasing  and continuous (at least until $W$ hits an endpoint of $I$) and we can define an inverse $\Gamma = A^{-1}$.
Finally set $Y_t = W_{\Gamma_t}$; then $Y$ is a diffusion in natural scale with speed measure $m$.

It follows that if $\tau$ is a solution of the SEP for $\nu$ in $W$ then $Q_\tau$ is a solution of the SEP for $\nu$ in $Y$. Similarly, if $\sigma$ is the solution of the SEP in $Y$, then $\Gamma_\sigma$ is a solution of the SEP in $W$.
Hence there is a one-to-one correspondence between solutions of the SEP for 
$\nu$ in $W$ and solutions for $\nu$ in $Y$.

Recall that we are supposing that $\nu \in L^1$. (Note that if $\nu \notin L^1$ then it is not possible to define $D_\nu(\cdot)$, and the Az\'ema-Yor solution is not defined.) Suppose also
that $w>\overline{\nu}$, which is the interesting case in which the Pedersen-Peskir and Cox-Hobson embeddings are distinct.
By analogy with (\ref{eqn:tauCH}) define
\begin{equation}
\label{eqn:tauY}
 \tau^Y_{CH,\nu} = \inf\{ u : Y_u \leq \beta_{\nu} (J^Y_u) \}
\end{equation}
where $\beta_\nu$ is as defined in (\ref{eqn:betadef}). Then $\tau = \tau^Y_{CH,\nu}$ inherits the embedding property from $\tau^W_{CH,\nu}$ and is a solution of the SEP for $\nu$ in $Y$.

Now consider the question of minimality. It is clear that $\tau^W_{CH,\nu}$ is minimal for $\nu$ in $W$ if and only if $\tau$ is minimal for $\nu$ in $Y$. If $\overline{\nu} \neq y$ then $\tau^W_{CH,\nu}$ is not integrable, but $\tau$ may be integrable. Further, if $\tau$ is integrable for $\nu$ in $Y$ started at $w$ and if $E_Y(w;\nu)<\infty$ then $\tau$ is minimal if and only if $\E[\tau] = E_Y(w;\nu)$.
In particular, if we choose the diffusion $Y$ so that its speed measure satisfies $m(\R)<\infty$, then necessarily $E_Y(w;\nu)<\infty$ (recall $\nu \in L^1$). The minimality of $\tau$ for $\nu$ in $Y$ and hence the minimality of $\tau^W_{AY,\nu}$ will follow if we can show $\E[\tau]=E_Y(w;\nu)$.

We have, (recall $w > \overline{\nu}$),
\begin{eqnarray*}
\E[\tau] & = & \int_w^\infty dz  \Prob(J^Y_\tau \geq z) \int_{\beta(z)}^z \frac{2(x - \beta(z))}{z-\beta(z)} m(dx) \\
         & = & 2 \int_{\R} \frac{b'(y)}{(b(y)-y)} dy \Prob(Y_\tau \geq y) \int_y^{b(y)} (x-y) m(dx) \\
         & = & 2 \int_{\R} \rho(y) dy \int_y^{b(y)} (x-y) m(dx) \\
         & = & 2 \int_{-\infty}^w m(dx) \int_{-\infty}^x (x-y) \rho(y) dy + 2 \int_w^\infty m(dx) \int_{\beta(x)}^x(x-y) \rho(y) dy.
         \end{eqnarray*}
Here we use excursion theory and
the fact that
\[ \E^x[H^Y_{a,b}] = 2 \int_a^b (x\wedge z-a)(b- x \vee z) m(dz) \hspace{10mm} a<x<b \]
for the first line (see also Pedersen and Peskir~\cite[Theorem 4.1]{PedersenPeskir:98}), $(J^Y_\tau \geq z)= (Y_\tau \geq \beta(z))$ for the second line, $b'(y) = \rho(y) (b(y)-y) / \nu((y,\infty))$ almost everywhere for the third, and the fact that $b(y) \geq w$ for the final line.

Observe that
\[ 2\int_{-\infty}^w m(dx) \int_{-\infty}^x (x-y) \nu(dy) = 2\int_{-\infty}^w \nu(dy) \int_{y}^w (x-y) m(dx) = \int_{-\infty}^w \nu(dy) q_{w}(y) . \]
Note that it is no longer true that $b = b_{\nu} =\beta_{\nu}^{-1}$ satisfies $b(y) = \E^{Y \sim \nu}[Y | Y \geq y]$ but rather $b(y) = \{(w - \overline{\nu}) + \int_y^\infty z \nu(dz) \}/(\nu(y,\infty))$ and then $(x - \beta(x)) \int_{\beta(x)}^\infty \nu(dz) = w - \overline{\nu} + \int_{\beta(x)}^\infty ( z - \beta(x)) \nu(dz)$. Thus
\[ \int_{\beta(x)}^x(x-y) \nu(dy) = \int_{\beta(x)}^\infty (x-y) \nu(dy) + \int_x^\infty (y-x) \nu(dy)  = (w - \overline{\nu}) + \int_x^\infty (y - x) \nu(dy), \]
and
\[ 2 \int_w^\infty m(dx) \int_{\beta(x)}^x(x-y) \nu(dy) = 2 ( w - \overline{\nu}) m((w,\infty)) + \int_{w}^\infty q_{w}(y) \nu(dy). \]

Finally then,
\[ \E[\tau] 
= 2 ( w - \overline{\nu}) m((w,\infty)) + \int q_{w}(y) \nu(dy) =
E_{Y}(w;\nu) \]
and hence $\tau$ and $\tau^W_{CH,\nu}$ given in (\ref{eqn:tauCH}) are minimal.

\subsection{An example}
In this example we suppose $Y$ is a non-negative, regular, local-martingale diffusion started at 1 with state space unbounded above and absorbed at zero (if $Y$ can hit zero in finite time, else $Y$ is assumed to be transient to zero). We suppose further that $\nu$ is the given by $\nu((y,\infty)) = (1 + \theta y)^{-\phi}$ with $\theta,\phi>0$ and $\phi \geq 1+1/\theta$. If $\phi = 1 + 1/\theta$ then $\overline{\nu}=1$, otherwise if $\phi>1 + 1/\theta$ then $\overline{\nu}<1$. (Note that if $\phi < 1 + 1/\theta$, then $\overline{\nu}>1$ and there is no embedding of $\nu$ in $Y$.)

Our first goal is to find the function $\beta_\nu$ in the Cox-Hobson extension of the Az\'ema-Yor embedding and the associated stopping times. In fact we find a family of solutions parameterised by $\psi \in [\overline{\nu},1]$ for which the stopping time 
with parameter $\psi$ corresponds to running $Y$ until it his $\psi$ and then embedding $\nu$ in $Y$ started at $\psi$ using the Cox-Hobson embedding. In particular this stopping time can be written as
\[ H^Y_\psi + \tau^\psi \circ \Theta_{H^Y_\psi} \]
where 
\[ \tau^\psi = \inf \{ u \geq 0 ; Y^\psi_u \leq \beta_{\nu, \psi} (J^{Y^\psi}_u) \} \]
and $Y^\psi$ satisfies $Y^\psi_0 = \psi$.
Here, for $\psi \in [\overline{\nu},1]$, $D_{\nu, \psi}(z) = \E^{Z \sim \nu}[(Z-x)^+] + (\psi - \overline{\nu})$ is given by
\[ D_{\nu, \psi}(z) = \psi  - \frac{1}{\theta(\phi-1)} \left\{ 1 - (1+\theta y)^{-(\phi - 1)} \right\} \]
and $b = \beta_{\nu,\psi}^{-1}$ given by (\ref{eqn:defbarycentre}) has expression
\[
 b(y) = (1+ \theta y)^\phi \left( \psi - \frac{1}{\theta(\phi-1)} \right) + \frac{\phi y}{\phi - 1} + \frac{1}{\theta(\phi-1)}  . \]

Now suppose $m(dy) = y^{-2c}dy$ (with $c \in (0,\infty) \setminus \{1/2,1 \}$) so that $Y$ solves $dY = Y^{c} dW$. Then $q = q_1$ is given by  $q(x) = \frac{x^{2-2c} - 1}{(1-c)(1-2c)} - \frac{2 (x-1)}{(1-2c)}$. We have
\[ E_Y(1;\nu) = \int_0^\infty  q(y) \nu(dy) + 2(1 - \overline{\nu}) m((1,\infty)) \]
Suppose $\phi > 1 + 1/\theta$.
Then $\overline{\nu}<1$ and there exists an integrable embedding of $\nu$ if and only if each of the three integrals
\[ \int^{\infty} x^{-2c} dx, \hspace{15mm} \int^\infty x^{2-2c} x^{-(\phi + 1)} dx, \hspace{15mm}
\int_0  x^{2-2c} dx  \]
is finite or equivalently $c>1/2$, $c>1 - \phi/2$ and $c<3/2$. However, since $\phi \geq 1 + 1/\theta>1$ this reduces to $1/2 < c < 3/2$.

If $\phi = 1 + 1/\theta$ then there is no requirement for $m((1,\infty))$ to be finite, the condition $c>1/2$ is not needed and there exists an integrable embedding of $\nu$ if and only if $1  - \phi/2 < c < 3/2$.

These statements are consistent with the case $c=0$ of absorbing Brownian motion. Then $\nu$ can be embedded in integrable time if an only if $\overline{\nu}=1$ and $\nu \in L^2$, or equivalently $\phi = 1+1/\theta$ and $\phi>2$.

\subsection{An example of Pedersen and Peskir}
Pedersen and Peskir~\cite{PedersenPeskir:01} give the expected time for a Bessel process to fall below a constant multiple of the value of its maximum, ie they find $\E[\tau^P_{AY}]$ where $\tau^P_{AY} = \inf \{ u > 0 : P_u \leq \lambda J^P_u \}$ and $\lambda < 1$. They find the answer by solving a differential equation subject to boundary conditions and a minimality principle.
We can recover their result directly using our methods.

Let $P$ be a Bessel process of dimension $\alpha \neq 2$, started at 1. Then $Y=P^{2-\alpha}$ is a diffusion in natural scale. Then $P$ solves $dY_t = (2-\alpha) Y_t^b dW_t$ where $b = (1-\alpha)/(2-\alpha)$.
Then $m(dy) = (2 -\alpha)^{-2} y^{-2b} dy$ and
\[ q_1(y) = \frac{1}{(2 - \alpha)^2} \left[ \frac{y^{2(1-b)}}{(1-b)(1-2b)} + \frac{1}{1-b} - \frac{2y}{1-2b} \right]. \]

Suppose first $\alpha<2$. We find, with $J_u = J^Y_u = \sup_{s \leq u} Y_u$,
\[ \tau^P_{AY} = \inf \{ u > 0 ;Y^{1/(2-\alpha)}_u \leq \lambda J_u^{1/(2 - \alpha)} \} = \inf \{ u > 0 ;Y_u \leq \gamma J_u \} =: \tau^\gamma \]
where $\gamma = \lambda^{2-\alpha}$.
Then, for $y \geq \gamma$,
\[
\Prob(Y_{\tau^{\gamma}} \geq y)  = \Prob(J_{\tau^\gamma} \geq y/\gamma)
 =  \exp \left( - \int_1^{y/\gamma} \frac{dj}{(j - \gamma j)} \right)
 =  (y/\gamma)^{-1/(1-\gamma)}.
\]

Then, if $\nu =\sL(Y_{\tau^\gamma})$ we have $\overline{\nu}=1$ and
\[ \E[ \tau^\gamma ] = \int_{\gamma}^\infty q_1(y) \nu(dy)
= \frac{\lambda^\alpha (2 -\alpha)}{\alpha(2 - \alpha \lambda^{\alpha - 2})} - \frac{1}{\alpha}\]
provided $\alpha \lambda^{\alpha - 2} < 2$, and otherwise $\tau^\gamma$ is not integrable.

If $\alpha > 2$ then set $Y = -P^{2-\alpha}$. Then $\tau^P_{AY} = \inf \{ u > 0 
: Y_u \leq \gamma {J}_u \} =: \tau^\gamma$ where $\gamma = 
\lambda^{2-\alpha}>1$. Then for $y \in (-\gamma,0)$, $\Prob(Y_{\tau^\gamma} \geq 
y) = (|y|/\gamma)^{1/(\gamma-1)}$. Again we find that $\nu \sim 
\sL(Y_{\tau^\gamma})$ has unit mean and
\[ \E[\tau^\gamma] = \frac{\lambda^\alpha(\alpha - 2)}{\alpha(\alpha \lambda^{\alpha - 2} - 2)} - \frac{1}{\alpha} ,\]
provided $\alpha \lambda^{\alpha - 2} > 2$,
else $\tau^{\gamma}$ is not integrable.

Finally, if $\alpha=2$, we set $Y = \log P$ and then $\tau^P_{AY} = \inf \{u>0: Y_u \leq J_u - \gamma \} =: \tau^\gamma$ where $\gamma = - \log \lambda > 0$. Then, for $y \geq - \gamma$,  $\Prob(Y_\tau \geq y) = e^{-(y/\gamma) - 1}$. Further, $dY_t = e^{-Y_t} dB_t$ and if $P_0=1$ then $Y_0=0$. Then $m(dy) = e^{2y} dy$ and
$q_0(y) = \{ e^{2y} - 2y - 1 \}/2$. Hence
\[ \E[ \tau^\gamma ] = \int_{-\gamma}^\infty q_0(y) \nu(dy) = \int_{-\gamma}^{\infty} \frac{e^{2y - (y/\gamma) - 1}}{2 \gamma} dy - \frac{1}{2} = \frac{\lambda^2}{2 + 4 \log \lambda} - \frac{1}{2} \]
provided $\lambda > e^{-1/2}$, and otherwise $\tau^\gamma$ is not integrable.

\bibliographystyle{plain}
\bibliography{general}

\end{document}